\documentclass[journal]{IEEEtran}
\IEEEoverridecommandlockouts                              
\bstctlcite{IEEEexample:BSTcontrol}                                                         

\usepackage{romannum}
\usepackage[usenames,dvipsnames]{xcolor}
\usepackage{tkz-berge}
\usetikzlibrary{fit,shapes}                                     
\usepackage{calrsfs}
\DeclareMathAlphabet{\pazocal}{OMS}{zplm}{m}{n}
\usepackage{graphicx}
\usepackage{graphics} 
\usepackage{epsfig} 
\usepackage{mathptmx}
\usepackage{caption}
\usepackage{subcaption}
\usepackage{times} 
\usepackage{tikz}
\usetikzlibrary{positioning}
\usepackage{amsmath, amssymb}

\usepackage[perpage]{footmisc}
\usepackage{amsthm}
\usepackage{amsfonts} 
\usepackage{setspace}
\usepackage{multirow}
\usepackage{textcomp}
\usepackage[english]{babel}
\usepackage{float} 
\usepackage{algorithmicx}
\usepackage[section]{algorithm}
\usepackage{algpascal}
\usepackage{algc}
\usepackage{algcompatible}
\usepackage{algpseudocode}
\let\oldReturn\Return
\renewcommand{\Return}{\State\oldReturn}
\usepackage{mathtools}
\usepackage{bm}
\usepackage{mathptmx}
\usepackage{times} 
\usepackage{mathtools, cuted}
\usepackage{textcomp}
\usepackage[english]{babel}
\usepackage{rotating}
\usepackage{pgfplots}
\pgfplotsset{compat=1.5}
\usepackage{siunitx}
\usepackage{pgfplotstable}
\usepackage{filecontents}
   \usepackage{authblk}
   \usepackage{etoolbox}

\newcommand{\norm[1]}{\left\lVert#1\right\rVert}

\newcommand{\T}{\scriptscriptstyle{T}}
\newcommand{\sB}{\scriptscriptstyle{B}}
\newcommand{\sD}{\scriptscriptstyle{D}}
\newcommand{\nT}{n_{\scriptscriptstyle T}}
\newcommand{\mT}{m_{\scriptscriptstyle T}}

\newcommand{\B}{\pazocal{B}}
\newcommand{\pT}{\pazocal{T}}
\newcommand{\sT}{\scriptscriptstyle {\pazocal{T}}}

\newcommand{\K}{\pazocal{K}}

\newcommand{\D}{\pazocal{D}}
\newcommand{\E}{\pazocal{E}}

\newcommand{\U}{{\bf u}}

\newcommand{\Q}{{\scriptstyle Q}}

\newcommand{\remove}[1]{}

\def \cN{{\mathcal N}}

\def \*{\star}
\def \10n{\!\!\!\!\!\!\!\!\!\!}

\newcommand{\R}{\mathbb{R}}
\newcommand{\bA}{\bar{A}}

\newcommand{\bB}{\bar{B}}

\newcommand{\bE}{\bar{E}}

\newtheorem{theorem}{Theorem}

\newtheorem{defn}[theorem]{Definition}
\newtheorem{cor}[theorem]{Corollary}
\newtheorem{rem}[theorem]{Remark}
\newtheorem{prob}[theorem]{Problem}

\newtheorem{prop}[theorem]{Proposition}
\numberwithin{theorem}{section}

                     \makeatletter
\makeatletter
\newcommand{\ssymbol}[1]{^{\@fnsymbol{#1}}}
\newcommand{\specificthanks}[1]{\@fnsymbol{#1}}
\makeatother

\title{\LARGE \bf Optimal Network Topology Design in Composite Systems with Constrained Neighbors for Structural Controllability}
\author{Shana~Moothedath\thanks{$^*$Department of Electrical and Computer Engineering,  University of Washington, USA. Email: sm15@uw.edu}\textsuperscript{\specificthanks{1}},
        Prasanna~Chaporkar\thanks{\textsuperscript{\textdagger}Department of Electrical Engineering, Indian Institute of Technology Bombay, India. Email: chaporkar@ee.iitb.ac.in}\textsuperscript{\specificthanks{2}},
        and~Aishwary~Joshi\thanks{$^\ddagger$Quadeye Securities, India. Email: aishwary@quadeyesecurities.com}\textsuperscript{\specificthanks{3}}
        }
\begin{document}
\maketitle
	
\begin{abstract}
 Composite systems  are large complex systems consisting  of interconnected  {\em agents} (subsystems).  Agents in a composite system interact with each other towards performing an intended goal. Controllability is essential to achieve  desired system performance in linear time-invariant  composite systems. Agents in a composite system are often  uncontrollable individually, further, only a few agents receive input. In such a case, the agents share/communicate their private state information with pre-specified neighboring agents so as to achieve controllability. Our objective in this paper is to identify an {\em optimal network topology}, optimal in the sense of minimum cardinality information transfer between agents to guarantee the controllability of the  composite system when the possible neighbor set of each agent is pre-specified.  We focus on graph-theoretic analysis referred to as {\em structural controllability} as numerical entries of system matrices in complex systems are mostly unknown. 
We first prove that given a set of agents and the possible set of neighbors, finding a minimum cardinality set of information (interconnections) that must be shared to accomplish structural controllability of the composite system is NP-hard. Subsequently, we present a polynomial-time algorithm that finds a $2$-optimal solution to this NP-hard problem. Our algorithm combines a minimum weight {\em bipartite matching} algorithm and a minimum {\em spanning tree} algorithm and gives a subset of interconnections which when established guarantees structural controllability, such that the worst-case performance is $2$-optimal. Finally, we show that   our approach directly extends to weighted constrained optimal network topology design problem and constrained optimal network topology design problem in switched linear systems.
\end{abstract}
\section{Introduction}\label{sec:intro}
  Real-world networks including robot swarms \cite{RafBay:10}, power grids \cite{Ter:11}, and biological systems \cite{Gu_stal:15}, consist of similar entities (subsystems) interacting with each other for performing a desired task. In order to achieve the intended system performance it is essential that the system is {\em controllable} \cite{Kai:80}.  
The exact numerical entries of the system matrices are often not known in complex systems \cite{LiuBar:16} and hence it is impossible to verify controllability of a system using Kalman's criteria \cite{Kai:80}. Graph theoretic analysis  referred to as {\em structural analysis} is widely used in the literature for analyzing controllability and related concepts of systems whose numerical system matrices are unknown and only the zero/nonzero pattern  of these matrices are known. The notion of {\em structural controllability} is the counterpart in structural analysis for controllability (see  \cite{LiuBar:16} for more details).  The strength of structural controllability is that using the topology (sparsity pattern) of the system it concludes Kalman's controllability of {\em almost all} systems of same sparsity pattern \cite{DioComWou:03}. 
We use  structural controllability to verify the controllability of the composite system.

  Complex networks consist of spatially distributed entities referred to as {\em agents} or {\it subsystems}. In many cases the individual subsystems are not structurally controllable on their own. Additionally, only a few of these subsystems called as {\em leaders}  receive input \cite{JafAjoAgh:11}.  To this end, the subsystems (agents) interact and share their private state information with neighboring subsystems to achieve structural controllability as a whole system. We refer to the interaction links  through which the subsystems share information as {\it interconnections} and the full system as the {\it composite system}.  Interconnections are directed edges between subsystems through which one subsystem communicates its state information with another subsystem. In this paper, we use subsystems and agents interchangeably.
 
  In most of the applications,  it is desired to achieve structural controllability  by keeping the amount of information transfer the least because of  security reasons, communication capacity constraints  to minimize the communication cost and delay, and power and resource constraints for battery operated agents. For instance, in multi-agent networked system, the robot swarm consists of many {\em agents}. In a formation control or consensus application with selected leaders, only few agents receive external input and hence each agent is not necessarily structurally controllable. Input addition to achieve structural controllability is not allowed as the input matrix is predefined, however, interconnection links (information transfer) between subsystems are often allowed. Thus for the robot swarm to be a controllable system,  the agents communicate their state informations with other agents \cite{JafAjoAgh:11}. While information sharing among agents are allowed, minimum information sharing is necessary for security reasons and to minimize delay and power consumption due to communication. Further, every agent can share information with pre-specified neighboring agents only.  In multi-agent systems and power applications, the neighbor set of each agent or subsystem is pre-specified depending on communication radius, security, and proximity. 

Our aim in this paper is to provide an algorithm that address the following questions in composite systems consisting of subsystems so that it achieves structural controllability:
\begin{itemize}
\item[(1)] Which neighboring agents should communicate?
\item[(2)] What state informations should be communicated?
\end{itemize}
In other words, the objective of this paper is to design an optimal network topology of the composite system by interconnecting neighboring agents such that the composite system is structurally controllable. We refer to this problem as the {\em optimal constrained network topology design problem}. 

\noindent The contributions of this paper are the following:
\begin{enumerate}
\item[$\bullet$] Given a set of subsystems and the possible set of neighbors of each subsystem,  we prove that the optimal constrained network topology design problem is NP-hard. The NP-hardness result is obtained from a polynomial-time reduction of a known NP-complete problem, the {\em degree constrained spanning tree problem} to the decision problem associated with the optimal constrained network topology design problem (Theorem~\ref{thm:NP}).
\item[$\bullet$] We provide a polynomial-time algorithm to find a solution to the optimal  constrained network topology design problem. This algorithm consists of a {\em minimum weight bipartite matching} algorithm and a {\em minimum spanning tree} algorithm (Algorithm~\ref{alg:twostage}).
 We prove that our algorithm is a $2$-optimal\footnote{An $\epsilon$-optimal algorithm is an algorithm whose solution value is at most $\epsilon$ times that of the actual optimum value.}  algorithm for the optimal constrained network topology design problem (Theorem~\ref{thm:approx}).
\item[$\bullet$] We prove that the computational complexity of approximation algorithm that solve optimal constrained network topology design problem is $O(\nT^{2.5})$, where $\nT$ is the dimension of the composite system (Theorem~\ref{thm:comp}).
\item[$\bullet$]  We show that   our approach directly extend to weighted constrained optimal network topology design problem and constrained optimal network topology design problem in switched linear systems (Corollaries~\ref{cor:approx_ext} and~\ref{cor:approx_switch}).
\end{enumerate}

The organization of the paper is as follows: in Section~\ref{sec:rel}, we present the related work in this area. In Section~\ref{sec:prob}, we formulate the optimal constrained network topology design problem. In Section~\ref{sec:prelim}, we present preliminaries and few existing results used in the sequel. In Section~\ref{sec:results}, we prove the main results of this paper, provide an approximation algorithm to solve the optimal constrained network topology design problem, and  illustrate our approximation algorithm using an example. In Section~\ref{sec:ext}, we discuss possible extensions of the results and algorithms. Finally, in Section~\ref{sec:conclu} we conclude the paper.
\section{Related Work}\label{sec:rel}
Controllability and observability of composite systems was introduced in \cite{Gil:63}, where it is related to the controllability and observability of its subsystems. Most of the earlier research in this area focus on standard interconnections, namely the series and the parallel connections \cite{CheDes:67}, \cite{WolHwa:74}, \cite{DavWan:75}. However, in practice the interconnections in large complex systems may not be only of these standard nature but will be complicated. Interconnections other than the standard ones are also considered in the literature, for instance see \cite{IkeSilYas:83}, \cite{Zho:15}.  Composite systems with subsystems of similar or identical dynamics is  addressed in \cite{Lun:86}, \cite{SunElb:91}.  Conditions for verifying system  theoretic properties of  composite systems made of identical subsystems are derived  in \cite{Lun:86}. Paper  \cite{SunElb:91} assumed that the subsystems are  symmetrically interconnected in addition to identical dynamics and derived conditions for verifying system  theoretic properties.  Specifically, controllability and observability of composite systems are addressed and conditions based on the subsystems are derived in papers \cite{CheDes:67},  \cite{WolHwa:74}, \cite{DavWan:75}, and \cite{Zho:15}. Decentralized controller design in composite systems is given in \cite{SunElb:91} when the  subsystems have identical dynamics and symmetric interconnections. 

Optimization problems in LTI composite systems are also studied in many papers for different problem settings.   The analysis in  \cite{RafBay:10}  yields an optimal network design for efficient average consensus of multi-agent systems. Consensus of multi-agent systems when the  communication graph of the agents have a spanning tree  is addressed \cite{LiDuaCheHua:10}. On the other hand, the approach in  \cite{DaiMaxMes:13} gave an optimal trajectory design  for establishing  connectivity of spatially distributed dynamic agents. Optimal  topology design problem is formulated in  \cite{GroStru:11}  when there is a trade-off between cost of communication links and the closed-loop performance and a solution approach is proposed by formulating it as a mixed-integer semi-definite programming problem.  While papers \cite{RafBay:10}, \cite{LiDuaCheHua:10},   \cite{DaiMaxMes:13}, and \cite{GroStru:11}   addressed the optimal network topology design problem for numerical systems, we perform our analysis  for structured systems. 
Our approach uses structural analysis for an optimal topology design in large complex network. Moreover, our focus is on the  structural controllability  of the network.

{\em Composite structured systems} is well studied in literature and various graph theoretic and algebraic conditions were derived for verifying structural controllability of composite systems in terms of subsystems \cite{Dav:77} -  \cite{LiXiZha:96}.  Specifically, the algorithm given in \cite{CarPeqAguKarJoh:17} accomplishes this using a distributed algorithm. In \cite{LiXiZha:96}, a graphic notion referred to as `g-cactus' is defined using which a sufficient condition is given for structural controllability of composite systems. Note that, all these papers focused on deriving conditions for verifying structural properties of the composite system using the subsystems. While optimal topology design is addressed for numerical systems in many papers,  we study this problem for structured composite systems.

Optimal constrained network topology design problem is addressed in   \cite{MooChaBel:17_homo} and \cite{MooChaBel:17_hetero} when the communication graph of subsystems are complete, i.e., every subsystem can possibly share information with any other subsystem. In such a case, the optimal constrained network topology design problem becomes unconstrained. Paper  \cite{MooChaBel:17_homo} consider the unconstrained version of the problem (i.e., communication graph of agents is complete) when the subsystems are  {\em homogeneous} or  so-called {\em structurally equivalent} and proposed a polynomial-time algorithm. Paper \cite{MooChaBel:17_hetero} considered the unconstrained case of the  optimal constrained network topology design problem when the subsystems are {\em heterogeneous} and gave a polynomial-time algorithm. Note that, the cases considered in  \cite{MooChaBel:17_homo} and \cite{MooChaBel:17_hetero} are polynomial-time solvable while the problem we consider in this paper is NP-hard. Hence the approaches in \cite{MooChaBel:17_homo} and \cite{MooChaBel:17_hetero} do not extend to the constrained case we address in this paper. 
\section{Problem Formulation}\label{sec:prob}
Consider an LTI system with dynamics $\dot{x} = Ax + Bu$, where $A \in \R^{n \times n}$ and $B \in \R^{n \times m}$. The structural representation of the matrices $A$ and $B$  are represented by $\bA \in \{0, \*\}^{n \times n}$ and $\bB \in \{0, \*\}^{n \times m}$, respectively. Here $\R$ denotes the set of real numbers and $\*$ denotes a free independent parameter. The pair $(\bA, \bB)$ structurally represents a system $(A, B)$ if it satisfies
\begin{eqnarray}\label{eq:struc}
A_{pq} &=& 0 \mbox{~whenever~} \bA_{pq} = 0,\mbox{~and} \nonumber \\
B_{pq} &=& 0 \mbox{~whenever~} \bB_{pq} = 0.
\end{eqnarray}
We refer to $(\bA, \bB)$ that satisfies \eqref{eq:struc} as the {\it structured system} representation of the {\it numerical system} $(A, B)$.  For a structured system,   structural controllability is defined as follows.

\begin{defn}\label{def:struccont}
The structured system $(\bA, \bB)$ is said to be structurally controllable if there exists at least one controllable numerical realization $(A, B)$.
\end{defn}
The matrix pair $(\bA, \bB)$ only indicates locations of zero and nonzero entries. For a given $(A, B)$, $(\bA, \bB)$ structurally represents a class of control systems corresponding to all possible numerical realizations of $(\bA, \bB)$.
\begin{rem}
While Definition~\ref{def:struccont} of structural controllability requires only one controllable realization, it is known that if a system is structurally controllable, then `almost all' numerical systems of the same sparsity structure is controllable  \cite{Rei:88}. In other words, structural controllability is a {\underline{ generic}} property.
\end{rem}
Now we describe structural representation of a composite system consisting of $k$ subsystems.
 Consider $k$ subsystems with structured state matrix $\bA_i \in \{0, \*\}^{n_i \times n_i}$ and structured input matrix  $\bB_i \in \{0, \*\}^{n_i  \times m_i}$,  for $i = 1, \ldots, k$.  The pair $(\bA_i, \bB_i)$ is referred as the $i^{\rm th}$ {\it subsystem} and is denoted by $S_i$.  With this notation, the dynamics of $S_i$ is 
\begin{eqnarray}
\dot{x}_i(t)& = &\bA_i x_i(t) + \bB_iu_i(t), \mbox{~for~} i=1,\ldots,k.
\end{eqnarray}
We do not assume that each subsystem is individually structurally controllable. To achieve structural controllability, one need to interconnect subsystems. Each subsystem can interconnect with only a pre-specified set of subsystems, referred as its {\em neighbors}. The set of neighbors of subsystem $S_j$ is denoted by the set $N(S_j)$.  We denote the {\it structured connection matrix} from $S_j$  to  $S_i$ by $\bE_{ij} \in \{0, \*\}^{n_i \times n_j}$. Then  $\bE_{ij} = 0$ if $S_i \notin N(S_j)$. Further,  $\bE_{ij} \neq 0$ implies that any state of $S_j$ can potentially connect to any state of $S_i$. Note that, neighbor relation is not symmetric and hence $S_i \in N(S_j) \nRightarrow S_j \notin N(S_i)$.
The  composite structured system of $k$ subsystems has the following dynamics. 
\begin{eqnarray*}\label{eq:sys}
\dot{x}(t)\hspace{-2 mm}&=&\hspace{-2 mm}\underbrace{\small{\begin{bmatrix}
\bA_1 & \bE_{12} & \cdots & \bE_{1k}\\
\bE_{21} & \bA_2 & \cdots & \bE_{2k}\\
\vdots & \ddots & \ddots & \vdots\\
\bE_{k1} & \bE_{k2} & \cdots & \bA_k
\end{bmatrix}}}_{\bA_{\T}}\hspace*{-1.0 mm}
x(t) +
\underbrace{\small{\begin{bmatrix}
\bB_1 & 0 & \cdots & 0\\
0 & \bB_2 & \cdots & 0\\
\vdots & \ddots & \ddots & \vdots \\
0 & 0 & \cdots & \bB_k
\end{bmatrix} }}_{\bB_{\T}}
u(t),
\end{eqnarray*}
where $\bA_{\T} \in \{0, \*\}^{\nT \times \nT}$ with $\nT =  \sum_{i = 1}^k n_i$ and $\bB_{\T} \in \{0, \*\}^{\nT \times \mT}$ with  $\mT = \sum_{i = 1}^k m_i$. Here, $x = [x_1^{\scriptscriptstyle T}, \ldots, x_k^{\scriptscriptstyle T}]^{\scriptscriptstyle T}$ with $x_i = [x_1^i,\ldots,x_{n_i}^i]^{\scriptscriptstyle T}$ and $u = [u_1^{\scriptscriptstyle T}, \ldots, u_k^{\scriptscriptstyle T}]^{\scriptscriptstyle T}$ with $u_i = [u_1^i,\ldots,u_{m_i}^i]^{\scriptscriptstyle T}$. The system $(\bA_{\T}, \bB_{\T})$ is said to be a {\em structured composite system} formed by subsystems $S_1, \ldots, S_k$  interconnected through $\bE_{ij}$'s, where $i,j \in \{1, \ldots, k\}$. 
 
Our aim in this paper is to design a structurally controllable optimal network topology of $(\bA_{\T}, \bB_{\T})$. Since we cannot  change the dynamics of the individual subsystem, optimality is with respect to designing interconnection matrices. Formally, the optimization problem we consider is as follows:

\begin{prob} \label{prob:int}
Given $k$ structured subsystems $(\bA_i, \bB_i)$, and the out-neighbor sets $N(S_i)$,  where $\bA_i  \in \{0,\*\}^{n_i \times n_i}$ and $\bB_i \in \{0, \*\}^{n_i \times m_i}$ and $i \in \{1,\ldots, k\}$, find \[\bA^\*_{\T}~ \in~ \arg\min\limits_{\10n~~ \bA_{\T} \in \K}  \norm[\bA_{\T}]_0,\]
where $\K :=\{\bA_{\T}\in \{0, \*\}^{\nT \times \nT}:\mbox{ for~all } i=1, \ldots, k, \mbox{~the~} (n_i \times n_i) \mbox{ diagonal submatrix of } \bA_{\T}$ is $\bA_i$, $\bE_{ij} \neq 0$ only if $S_i \in N(S_j)$, and  $(\bA_{\T}, \bB_{\T})$  is structurally  controllable$\}$. 
\end{prob}

Here, $\norm[\cdot]_0$ denotes the zero matrix 
   norm\footnote{Although $\norm[\cdot]_0$ does not satisfy some of 
   the norm axioms, the number of non-zero entries in a matrix is 
   conventionally referred to as the {\em zero norm}.}. The set $\K$ denotes the set of all feasible solutions of Problem~\ref{prob:int}.  Without loss of generality, we assume that the matrix  $\bB_{\T}$ is nonzero.  Further, we assume that for $\bE_{ij}=\{\*\}^{n_i \times n_j}$ for all $i, j$ satisfying $S_i \in N(S_j)$, the composite structured system is structurally controllable. In other words, when the subsystems are composed with all possible interconnections, the  resulting composite system is structurally controllable and  hence the set  $\K$ is non-empty. Two matrices $\bA'_{\T}$ and $\bA''_{\T}$ in $\K$ differs only in their off-diagonal blocks.  Solving the minimum interconnection problem is same as minimizing the non-zero entries in matrices in $\K$, since for all matrices in $\K$ the diagonal blocks are fixed and  optimization is possible only corresponding to the off-diagonal blocks. This in turn is same as minimizing the interconnections. 

\section{Preliminaries}\label{sec:prelim}
In this section, we detail the preliminaries and introduce few existing results in structural analysis used in the sequel. Most of the existing work in structural analysis is based on graph theoretic analysis  as the interaction of states, inputs, and outputs in a structured system are well captured  in graphs.

Consider a structured  system\footnote{Typical structured system is denoted by $(\bA, \bB)$ and the related concepts discussed in this section can be extended to specific system under consideration.} $(\bA, \bB)$, where $\bA \in \{0, \*\}^{n \times n}$ and $\bB \in \{0, \*\}^{n \times m}$. We first construct the {\em state digraph} $\D(\bA)$ with vertex set $V_{X}= \{x_1, \ldots, x_n\}$ and edge set $E_{X}$, where $(x_p, x_q) \in E_{X}$ if $\bA_{qp} = \*$. Now we construct {\em subsystem digraph} $\D(\bA, \bB)$, with vertex set $V_{X} \cup V_{U}$ and edge set $E_{X} \cup E_{U}$. Here, $V_{U} = \{ u_1, \ldots, u_{m}\}$ and $(u_p, x_q) \in E_{U}$ if  $\bB_{qp} = \*$. All state vertices  are said to be {\em accessible} in $\D(\bA, \bB)$ if there exists a directed path to every state node $x_q$ from some input node $u_p$  in $\D(\bA, \bB)$.
Subsequently, we construct the  {\em state bipartite graph} $\B(\bA)$  with vertex set  $(V_{X'} \cup V_{X})$ and edge set $\E_{X}$. Here, $V_{X'}=\{{x'}_1, \ldots {x'}_{n}\}$ and $({x'}_q, x_p) \in \E_{X} \Leftrightarrow ({x}_p, x_q) \in E_{X}$. The {\em subsystem bipartite graph} $\B(\bA, \bB)$ is built with vertex set $(V_{{X'}}, V_{X}~ \cup ~V_{U})$ and edge set $\E_{X} \cup \E_{U}$, where $({x'}_p, u_q) \in \E_{U} \Leftrightarrow (u_q, x_p) \in E_{U}$. 

Now we define the notion of {\em perfect matching} in a bipartite graph. An undirected graph $G_{\sB} = (V_{\sB}, \widetilde{V}_{\sB}, \E_{\sB})$ is said to be a bipartite graph if $V_{\sB} \cap \widetilde{V}_{\sB} = \emptyset$ and $\E_{\sB} \subseteq V_{\sB} \times \widetilde{V}_{\sB} $. A matching $M_{\sB}$ in $G_{\sB}$ is a subset of edges such that $M_{\sB} \subseteq \E_{\sB}$ and for $(\sigma, \beta), (\lambda, \mu) \in M_{\sB}$ we have $\sigma \neq \lambda$ and $\beta \neq \mu$, where $\sigma, \lambda \in V_{\sB}$ and $\beta, \mu \in \widetilde{V}_{\sB} $. Now we give the classical result for structural controllability.

\begin{prop}[pp.207, \cite{Lin:74}]\label{prop:lin}
A structured system $(\bA, \bB)$ is said to be structurally controllable if and only if all state nodes are accessible in $\D(\bA, \bB)$ and the bipartite graph $\B(\bA, \bB)$ has a perfect matching, i.e., no dilations in $\D(\bA, \bB)$.
\end{prop}

The bipartite graph $\B(\bA, \bB)$ that has a perfect matching is said to satisfy the {\em no-dilation} condition. The subsystems $(\bA_i, \bB_i)$ need not be individually structurally controllable\footnote{The subsystem $S_i = (\bA_i, \bB_i)$ is not structurally controllable if state nodes in the digraph $\D(\bA_i, \bB_i)$ are not accessible and/or the bipartite graph $\B(\bA_i, \bB_i)$ does not have a perfect matching.}, however, one can achieve structural controllability by composing the subsystems using interconnections. The objective of this paper is to achieve it using minimum number of interconnections.
\section{Main Results: Complexity and Approximation }\label{sec:results}
In this section, we first prove that Problem~\ref{prob:int} is NP-hard and then present a polynomial time approximation algorithm that gives a $2$-optimal solution.
\subsection{Complexity Results}
In this subsection, we analyze the complexity of Problem~\ref{prob:int}. Specifically, we show that Problem~\ref{prob:int} is NP-hard.

\begin{theorem}\label{thm:NP}
Consider $k$ structured subsystems $(\bA_i, \bB_i)$, where $\bA_i  \in \{0,\*\}^{n_i \times n_i}$ and $\bB_i \in \{0, \*\}^{n_i \times m_i}$, and the neighbor set $N(S_j)$, for $i,j  \in \{1,\ldots, k\}$. Then, Problem~\ref{prob:int} is NP-hard.
\end{theorem}
The proof of Theorem~\ref{thm:NP} is established using reduction of general instance of a known NP-complete problem, namely {\em degree constrained spanning tree problem} (DCST), to an instance of the decision problem associated with Problem~\ref{prob:int}. We describe below the DCST problem.
\begin{prob}[Degree constrained spanning tree problem \cite{GarJoh:02}]
Given an undirected graph $G=(V, E)$ and a positive integer $\gamma$, where $\gamma \leqslant |V|$, does there exist a spanning tree such that no node in the spanning tree has degree $\geqslant \gamma~ ?$ 
\end{prob}

Note that, $G$ is an undirected {\em connected} graph. The DCST problem  is shown to be NP-complete fo any $ \gamma \geqslant 2$ \cite{GarJoh:02}. When $\gamma=2$, the problem reduces to a Hamiltonian path problem \cite{GarJoh:02}, a known NP-complete problem. We now present the decision problem associated with Problem~\ref{prob:int}. \\
\begin{prob}[Decision version of Problem~\ref{prob:int}]\label{prob:int_decision}
Given $k$ structured subsystems $S_i = (\bA_i, \bB_i)$, the out-neighbor set $N(S_i)$, and  a positive integer $\alpha$, where $\bA_i  \in \{0,\*\}^{n_i \times n_i}$, $\bB_i \in \{0, \*\}^{n_i \times m_i}$ and $i  \in \{1,\ldots, k\}$, does there exist  $\bA'_{\T} \in \K$ such that the number of interconnections in $\bA'_{\T}$ is less than or equal to $\alpha$? 
\end{prob}
\noindent Now we prove that Problem~\ref{prob:int_decision} is NP-complete.
\begin{theorem}\label{thm:NP_decision}
Consider $k$ structured subsystems $S_i = (\bA_i, \bB_i)$, the out-neighbor sets $N(S_i)$, and a positive integer $\alpha$, where $\bA_i  \in \{0,\*\}^{n_i \times n_i}$, $\bB_i \in \{0, \*\}^{n_i \times m_i}$, and $i \in \{1,\ldots, k\}$. Then, Problem~\ref{prob:int_decision} is NP-complete.
\end{theorem}
\begin{proof}
We prove  Problem~\ref{prob:int_decision} is NP-complete using a polynomial time reduction of DCST problem with $\gamma=2$ (Hamiltonian path problem) to Problem~\ref{prob:int_decision}. Firstly, consider a general instance of the DCST problem with undirected graph $G=(V, E)$ and $\gamma=2$, where $|V| = r$.   Subsequently, construct $r$ instances of  Problem~\ref{prob:int_decision}, say $P_1, P_2, \ldots, P_r$. Consider the instance $P_j$. Associate with each vertex $v_i \in V$  a subsystem $(\bA_i, \bB_i)$ and the out-neighbor sets $N(S_i)$, for $i=1, \ldots, k$. Since $|V| = r$, we get $k=r$. Further, all subsystems have the same structured  state matrix $\bA_i $ and only the $j^{\rm th}$ subsystem receive input (for instance $P_j$). Specifically, $\bA_i = \left[\begin{smallmatrix}
0 & \* & 0\\
\* & 0 & \*\\
0 & \* & 0 
\end{smallmatrix} \right]$, for $i=1, \ldots r$, $\bB_j =[\*, 0, 0]^{\T}$, and $\bB_i =0$, for $i\in \{1,2, \ldots, r\}$ and $i \neq j$. The neighbor sets are defined such that if $(v_i, v_j) \in E$, then  $S_j \in N(S_i)$. Figure~\ref{fig:illus_1} shows the set of $r$ subsystems constructed from  $G=(V, E)$.

Now we show that DCST problem on $G$ with $\gamma=2$ has a solution if and only if there exists a solution to Problem~\ref{prob:int_decision} for at least one of the $r$ instances, $P_1, P_2, \ldots, P_r$, of Problem~\ref{prob:int_decision} with $\alpha = r-1$. This proves Theorem~\ref{thm:NP_decision} due to the fact that if Problem~\ref{prob:int_decision} is polynomial-time solvable, then the total complexity of solving $r$ instances of Problem~\ref{prob:int_decision} is polynomial.
\\
{\bf If part:} Here we show that if there exists a solution to Problem~\ref{prob:int_decision} on one of the constructed instances $P_1, \ldots, P_r$ with $\alpha = r-1$, say $(\bA^\*_{\T}, \bB_{\T})$,  then there exists a solution to the DCST problem on $G$ with $\gamma=2$. Without loss of generality, assume that there exists a solution to instance $P_1$ of Problem~\ref{prob:int_decision}.   Note that  the subsystems $S_1, \ldots, S_r$ constructed here are such that $S_1:=(\bA_1, \bB_1)$ is structurally controllable, and for $i=2, \ldots, r$, subsystem $S_i:=(\bA_i, \bB_i)$ is inaccessible as $\bB_i =0$ and there exists one dilation, i.e., size of maximum matching in $\B(\bA_i, \bB_i)$ is $n_i - 1$. As the composite system obtained is structurally controllable using $(r-1)$ number of interconnections,  by construction every subsystem is interconnected to at most two subsystems. Construct a graph $\pazocal{T} =(V_{\pazocal{T}}, E_{\pazocal{T}})$, where $ V_{\pazocal{T}} = V$ and $E_{\pazocal{T}} \subset E$ such that $(v_i, v_j) \in E_{\pazocal{T}}$ if directed interconnection edge $(x_p^i, x_q^j)$ is present in the composite system $(\bA^\*_{\T}, \bB_{\T})$, where $p,q \in \{1,2,3\}$. Since each subsystem is connected to at most two other subsystems in $(\bA^\*_{\T}, \bB_{\T})$, $\pazocal{T}$ is a spanning tree with  node-degree less than or  equal to $2$. This complete the if-part.
\\
{\bf Only-if part:} Here we show that if there exists a solution to the DCST problem on $G$ with $\gamma=2$,  then there exists a solution to at least one of the instance of $P_1, \ldots, P_r$ of Problem~\ref{prob:int_decision} on the  constructed subsystems and out-neighbor set for $\alpha = r-1$. Let $\pT = (V_{\pT}, E_{\pT})$, where $V_{\pT} = V$, $E_{\pT} \subset E$ and $|E_{\pT}|=r-1$, be a spanning tree
(path) of $G$ with node-degree $\leqslant 2$.  Also let the vertices in the spanning tree are ordered  such that $\pT := \{v_{i_1}, v_{i_2}, \ldots, v_{i_r} \}$,  $(v_{i_a}, v_{i_b}) \in E_{\pT}$ implies $a<b$, where $\cup_{a=1}^rv_{i_a} = V$. Note that $v_{i_1}$ and $v_{i_r}$ are the leaf nodes of $\pT$. Choose one of the leaf vertices of $\pT$, say $v_{i_1}$.  Now consider the ${i_1}^{\rm th}$ instance of Problem~\ref{prob:int_decision}, i.e., $P_{i_1}$. Recursively select edge $(v_{i_a}, v_{i_b}) \in E_{\pT}$ and compose subsystems $S_{i_a}$ and $S_{i_b}$ of $P_{i_1}^{~\rm th}$ instance of Problem~\ref{prob:int_decision}. Interconnect subsystems $S_{i_a}$ and $S_{i_b}$ through directed interconnection edge $(x_3^{i_a}, x_1^{i_b})$. Notice that the first interconnection edge made here is $(x_3^{i_1}, x_1^{i_2})$ and the composite system formed by subsystems $S_{i_1}$ and $S_{i_2}$  using $(x_3^{i_1}, x_1^{i_2})$ is structurally controllable (since $\bB_{i_1} \neq 0$ guarantees accessibility and $(x_3^{i_1}, x_1^{i_2})$ removes dilation). As $\pT$ is a spanning tree with node-degree $\leqslant 2$, this recursive process terminates in exactly $r-1$ steps such that we compose all subsystems  to form a structurally controllable composite system with $r-1$  number of interconnections.
This proves  the only-if part.
The if and the only-if part together prove that Problem~\ref{prob:int_decision} is NP-complete.
\end{proof}
\noindent Theorem~\ref{thm:NP_decision} now proves one of the main result of this paper.
{\bf Proof~of~Theorem~\ref{thm:NP}:} As the decision version (Problem~\ref{prob:int_decision}) is NP-complete, the optimization problem, Problem~\ref{prob:int}, is NP-hard \cite{GarJoh:02}. Thus proof of Theorem~\ref{thm:NP} directly follows from Theorem~\ref{thm:NP_decision}.
\qed
\begin{figure}[h]
\centering
\begin{tikzpicture}[->,>=stealth',shorten >=1pt,auto,node distance=1.85cm, main node/.style={circle,draw,font=\scriptsize\bfseries}]
\definecolor{myblue}{RGB}{80,80,160}
\definecolor{almond}{rgb}{0.94, 0.87, 0.8}
\definecolor{bubblegum}{rgb}{0.99, 0.76, 0.8}
\definecolor{columbiablue}{rgb}{0.61, 0.87, 1.0}

  \fill[bubblegum] (5,1.0) circle (6.0 pt);

  \fill[almond] (5.0,0) circle (6.0 pt);
  \fill[almond] (5.75,0) circle (6.0 pt);
  \fill[almond] (6.5,0) circle (6.0 pt);
  
  \fill[almond] (7.5,0) circle (6.0 pt);
  \fill[almond] (8.25,0) circle (6.0 pt);
  \fill[almond] (9.0,0) circle (6.0 pt);
  \fill[almond] (11.0,0) circle (6.0 pt);
   \fill[almond] (11.75,0) circle (6.0 pt);
    \fill[almond] (12.5,0) circle (6.0 pt);

   \node at (5,1.0) {\small $u^1_1$};     

  \node at (5,0) {\small $x_1^1$};
  \node at (5.75,0) {\small $x_2^1$};
  \node at (6.5,0) {\small $x_3^1$};

  \node at (7.5,0) {\small $x_1^2$};
  \node at (8.25,0) {\small $x_2^2$};
  \node at (9.0,0) {\small $x_3^2$};
  \node at (11.0,0) {\small $x_1^r$};
    \node at (11.75,0) {\small $x_2^r$};
      \node at (12.5,0) {\small $x_3^r$};
      
      \node at (10,0) {\small $\ldots$};

\draw (5.0,0.75)  ->   (5.0,0.25);  

\path[every node/.style={font=\sffamily\small}]
(7.5,0.25) edge[bend left = 40] node [left] {} (8.25,0.25)
(8.25,-0.25) edge[bend left = 40] node [left] {} (7.5,-0.25)
(8.25,0.25) edge[bend left = 40] node [left] {} (9,0.25)
(9,-0.25) edge[bend left = 40] node [left] {} (8.25,-0.25)

(11,0.25) edge[bend left = 40] node [left] {} (11.75,0.25)
(11.75,-0.25) edge[bend left = 40] node [left] {} (11,-0.25)
(11.75,0.25) edge[bend left = 40] node [left] {} (12.5,0.25)
(12.5,-0.25) edge[bend left = 40] node [left] {} (11.75,-0.25)

(5,0.25) edge[bend left = 40] node [left] {} (5.75,0.25)
(5.75,-0.25) edge[bend left = 40] node [left] {} (5,-0.25)
(5.75,0.25) edge[bend left = 40] node [left] {} (6.5,0.25)
(6.5,-0.25) edge[bend left = 40] node [left] {} (5.75,-0.25);

\node (rect) at (5.85,0.4) [draw,dashed,minimum width=2.25cm,minimum height=1.8cm] {};
\node at (5.8, 1) {$S_1$};
\node (rect) at (8.25,0.4) [draw,dashed,minimum width=2.1cm,minimum height=1.8cm] {};
\node at (8.25, 1) {$S_2$};
\node (rect) at (11.75,0.4) [draw,dashed,minimum width=2.1 cm,minimum height=1.8cm] {};
\node at (11.75, 1) {$S_r$};
\end{tikzpicture}
\caption{\small Structured subsystems $S_1, S_2, \ldots, S_r$ corresponding to instance $P_1$ constructed from an undirected  graph $G=(V, E)$ with $|V| =r$ for proving NP-completeness of Problem~\ref{prob:int_decision}.}
\label{fig:illus_1}
\end{figure}
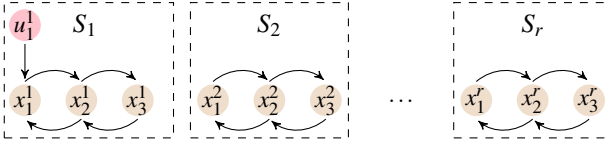
\begin{rem}
Problem~\ref{prob:int_decision} is NP-complete and Problem~\ref{prob:int} is NP-hard 
 even when the composite system is single-input with   irreducible and homogeneous ($\bA_1=\bA_2 = \cdots = \bA_k$) subsystems as the instance constructed in the proofs belongs to this class.
 \end{rem}
\subsection{Approximation Algorithm}\label{subsec:approx}
In this subsection, we present a polynomial time $2$-optimal algorithm to solve Problem~\ref{prob:int}. The proposed algorithm includes a minimum weight bipartite matching algorithm
and a minimum spanning tree algorithm
\cite{Die:00}. We first define minimum weight bipartite matching and minimum spanning tree below and then give the approximation algorithm.

\begin{defn}
In a bipartite graph $G_{\sB} = (V_{\sB}, \widetilde{V}_{\sB}, \E_{\sB})$ with $|V_{\sB}| \leqslant |\widetilde{V}_{\sB}|$ and a weight function $w: \E_{\sB} \rightarrow \R$, a minimum weight bipartite matching is a matching $M_{\sB}$ such that $|M_{\sB}| = |V_{\sB}|$ and $\sum_{e \in M_{\sB}} w_{\sB}(e) \leqslant \sum_{e \in \widetilde{M}_{\sB}} w_{\sB}(e)$ for all matching $\widetilde{M}_{\sB}$ satisfying $|\widetilde{M}_{\sB}|  = |V_{\sB}|$.
\end{defn}
\begin{defn}
In a directed graph $G_{\sD} = (V_{\sD}, E_{\sD})$ with weight function $c: E_{\sD} \rightarrow \R$, a directed spanning tree rooted at $r \in V_{\sD}$ is a subgraph $T$ such that the undirected version of $T$ is a tree and there exists a directed path from $r$ to any vertex in $V_{\sD}$. Cost of spanning tree $T$ is defined as $\sum_{e \in T}c(e)$ and minimum spanning tree is one with minimum cost.
\end{defn}
The pseudo-code of the proposed algorithm is given in Algorithm~\ref{alg:twostage}. We first construct a bipartite graph $\hat{\B}(\bA_{\T}, \bB_{\T}) $ whose left vertex set is the state nodes of all $k$ subsystems, i.e., $\hat{V}_{\T'}=\cup_{i=1}^k V_{X'_i}$ (Step~\ref{step:Vx}) and right vertex set  is state and input nodes of all $k$ subsystems, i.e., $\hat{V}_{\T} = \cup_{i=1}^k (V_{X_i} \cup V_{U_i})$ (Step~\ref{step:V'x}). The edge set $\hat{\E}_{\T}$ consists of two types of edges: (i)~edges corresponding to nonzero entries in $\bA_i, \bB_i$, for $i=1,\ldots, k$, denoted by $\hat{\E}_{S}$, and (ii)~edges  corresponding to possible interconnections between subsystems  denoted by $\hat{\E}_{N}$, based on the pre-specified out-neighbor set (Step~\ref{step:Ex}). Now we define a weight  function $w_{\sB}: \hat{\E}_{\T} \rightarrow \R$. The weight $w_{\sB}$ assigns zero weight to all edges within a subsystem and assigns weight $1$ to all edges corresponding to interconnections (Step~\ref{step:weight}). A minimum weight perfect matching in $\hat{\B}(\bA_{\T}, \bB_{\T}) $ is denoted by $M_A$ (Step~\ref{step:match}) and the edges in $\hat{\E}_{N}$ that are present in $M_A$ are denoted by $\hat{\E}_{N}(M_A)$ (Step~\ref{step:edgematching}). A minimum weight matching in $\hat{\B}(\bA_{\T}, \bB_{\T}) $ under  weight $w_{\sB}$ gives  a minimum cardinality subset of interconnection edges that will ensure existence of a  perfect matching in the composite system, thereby guaranteeing the no-dilation condition. This completes Stage~1.
\begin{algorithm}[h]
  \caption{Pseudo-code for solving Problem \ref{prob:int} using minimum weight maximum matching and a weighted directed spanning tree algorithm
  \label{alg:twostage}}
  \begin{algorithmic}
\State \textit {\bf Input:} Structured subsystems $S_i =(\bA_i, \bB_i)$ and out-neighbor set $N(S_i)$, for $i=1,\ldots, k$
\State \textit{\bf Output:} Interconnection edges that achieves structural controllability of the composite system
\end{algorithmic}
\hspace*{-2mm}\begin{minipage}{1cm}
\[
\rotatebox{90}{\hspace*{-4mm}\small Stage 1\hspace*{4mm}}\left\{ \vphantom{ \begin{array}{c} 5\\5\\5\\5\\5 \\ 5\\5\\5\\5 \\5\\5\\5\\5\\5  \\ \end{array} } \right.
\]
\vspace*{-2mm}
\[
\rotatebox{90}{\hspace*{-4mm}\small Stage 2\hspace*{2mm}}\left\{ \vphantom{ \begin{array}{c} 5 \\5\\5\\5\\5\\5\\5\\5\\5\\5\\5\\5\\5\\5\\5\\5\\5\\5 \\5 \\[-1mm] 5    \end{array} } \right.
\] 
\end{minipage}\\[-145mm]
\hspace*{4mm}
\begin{minipage}{8cm}
\begin{algorithmic}[1]
\State  Construct bipartite graph $\hat{\B}(\bA_{\T}, \bB_{\T}) = (\hat{V}_{\T'}, \hat{V}_{\T}, \hat{\E}_{\T})$\label{step:bip}
\State Left vertex set $\hat{V}_{{\T}'} = \cup_{i=1}^k V_{X'_i}$, where $V_{X'_i} = \{x'^i_1, \ldots, x'^i_{n_i} \}$\label{step:Vx}
\State  Right  vertex set  $\hat{V}_{\T} = \cup_{i=1}^k (V_{X_i} \cup V_{U_i})$, where $V_{X_i} = \{x'^i_1, \ldots, x^i_{n_i} \}$ and $V_{U_i} = \{u^i_1, \ldots, u^i_{m_i} \}$\label{step:V'x}
\State Edge set  $\hat{\E}_{\T} := \hat{\E}_S \cup \hat{\E}_{N}$, where $({x'^i}_p, x_q^i) \in \hat{\E}_S$ if $\bA_{i_{pq}}=\*$, $({x'^i}_p, u_q^i) \in \hat{\E}_S$ if $\bB_{i_{pq}}=\*$, and 
$({x'^i}_p, x_q^j)\in \hat{\E}_N$ if $S_i \in N(S_j)$, for $i,j \in \{1,\ldots, k\}$, $i \neq j$, and $ p,q \in \{1, \ldots, n_i\}$\label{step:Ex}
\State $w_{\sB}(e) \leftarrow \begin{cases}
0 , ~~ {\rm for}~ e \in \hat{\E}_S,\\
1,~~ {\rm for}~ e \in \hat{\E}_N. 
\end{cases}$ \label{step:weight}
\State Find minimum weight perfect matching of $\hat{\B}(\bA_{\T}, \bB_{\T})$ under weight function $w_{\sB}$, say $M_A$\label{step:match}
\State Edge set $\hat{E}_N(M_A) := \{ (x^j_q, x^i_p) : (x'^i_p, x^j_q) \in M_A \cap  \hat{\E}_N\}$\label{step:edgematching}
\State Construct the directed graph $\pT_{\cN} = (V_{\cN}, E_{\cN})$\label{step:span_graph}
\State $V_{\cN} := \cN_{S} \cup \{\U\}$, where $\cN_{S} = \{\cN_1, \ldots, \cN_{\Q}\}$ is the set of  SCC's of subsystems $S_1, \ldots, S_k$ and $\U$ is a master input node\label{step:span_graph_node}
\State $E_{\cN} := E^1_{\cN} \cup E^2_{\cN} \cup E^3_{\cN}$\label{step:fulledge}
\State $E^1_{\cN}\leftarrow\{ (\cN_g, \cN_h): x^i_p \in \cN_g,  x^i_q \in \cN_h \mbox{~and~} (x^i_p , x^i_q) \in E_{X_i} \mbox{~for~} p,q \in \{1, \ldots, n_i\},  i \in \{1, \ldots, k\} \}$\label{step:fulledge1}
\State $E^2_{\cN}  \leftarrow \{ (\cN_g, \cN_h): x^i_p \in \cN_g,  x^j_q \in \cN_h \mbox{~and~} S_j \in N(S_i) \mbox{~for~} p,q \in \{1, \ldots, n_i\},  i,j \in \{1, \ldots, k\}, i \neq j \}$\label{step:fulledge2}
\State $E^3_{\cN}\leftarrow \{ (\U, \cN_h):  x^i_q \in \cN_h \mbox{~and}  \mbox{~node~} x^i_q \mbox{~is}\mbox{~accessible~in}~ \D(\bA_i, \bB_i ), \mbox{~for~} i=1, \ldots, k\}$\label{step:fulledge3}
\State $w_{\sT}(e) \leftarrow \begin{cases}
0 , ~~ {\rm for}~ e \in E^1_{\cN},\\
1,~~ {\rm for}~ e \in E^2_{\cN},\\
0 , ~~ {\rm for}~ e \in E^3_{\cN}. 
\end{cases}$ \label{step:weight2}
\State Find minimum weight spanning tree of $\pT_{\cN}$ under weight function $w_{\sT}$, say $T_A$\label{step:tree}
\State Edge set $E_{\cN}(T_A) := \{e: e\in T_A \cap E^2_{\cN} \}$\label{step:edgetree1}
\State Edge set $\widetilde{E}_{\cN}(T_A) := \{ (x^i_p , x^j_q) : x^i_p \in \cN_g, x^j_q \in \cN_h \mbox{~and~} (\cN_g, \cN_h) \in  E_{\cN}(T_A)\}$\label{step:edgetree2}
\State Final set of interconnections is $\hat{E}_N(M_A) \cup \widetilde{E}_{\cN}(T_A)$
\end{algorithmic}
\end{minipage}
\end{algorithm}

In Stage~2, we construct a directed graph $\pT _{\cN} =(V_{\cN}, E_{\cN})$ (Step~\ref{step:span_graph}). The node set of  $\pT _{\cN}$ consists of strongly connected components\footnote{A strongly connected component of a directed graph is a maximal subgraph of the graph in which there exists a directed path between any two distinct vertices.} (SCC's) of the subsystems and a {\em master input node} denoted by $\U$ (Step~\ref{step:span_graph_node}). Node $\U$ corresponds to all input nodes in subsystems $S_1, \ldots, S_k$. The edge set $E_{\cN}$  is partitioned into three categories: (a)~edges between SCC's in the same subsystem, (b)~edges from SCC's of one subsystem to SCC's of its out-neighbors, and (c)~edges from $\U$ to SCC's of subsystems  which contain state nodes that are accessible in their respective digraphs $\D(\bA_i, \bB_i)$, for $i=1, \ldots, k$. Type~(a) corresponds to directed edges in $\D(\bA_i)$ that connect state nodes in two SCC's of $\D(\bA_i)$. Type~(b) consists of edges between all SCC's of subsystems  $S_i$ and $S_j$ if $S_i \in N(S_j)$. Type~(c) consists of edges that correspond to entries in $\bB_i$'s. These are associated with $\*$ entries in $\bB_i$ that connects an input to some state node in SCC's of subsystems.  A minimum weight spanning tree $T_A$ of $\pT_{\cN}$ rooted at $\U$ returns a set of edges $\widetilde{E}_{\cN}(T_A) $, which corresponds to the minimum number of interconnections essential to make the composite system accessible. Now we give the next main result of this paper.
\begin{theorem}\label{thm:approx}
Algorithm~\ref{alg:twostage} that takes as input $k$ structured subsystems $(\bA_i, \bB_i)$, where $\bA_i  \in \{0,\*\}^{n_i \times n_i}$ and $\bB_i \in \{0, \*\}^{n_i \times m_i}$, and the neighbor set $N(S_i)$, for $i  \in \{1,\ldots, k\}$ returns a set of interconnections such that the composite system obtained using the interconnections is a $2$-optimal solution to Problem~\ref{prob:int}.
\end{theorem}
\begin{proof}
The interconnections obtained as output of Algorithm~\ref{alg:twostage} is $\hat{E}_N(M_A) \cup \widetilde{E}_{\cN}(T_A)$. Here, $\hat{E}_N(M_A)$ is an interconnection set of minimum cardinality that guarantee the no-dilation condition of the composite system (since these edges corresponds to a minimum weight matching). Similarly, $\widetilde{E}_{\cN}(T_A)$ is an interconnection set of minimum cardinality that guarantee the accessibility of the composite system. Hence $\hat{E}_N(M_A) \cup \widetilde{E}_{\cN}(T_A)$ guarantees accessibility and no-dilation and the composite system obtained using the interconnection set $\hat{E}_N(M_A) \cup \widetilde{E}_{\cN}(T_A)$ is structurally controllable. Let the minimum number of interconnections required for structural controllability of the system is denoted as $\Delta$. Then
\begin{eqnarray}
\Delta &\geqslant & |\hat{E}_N(M_A)|,\nonumber\\
\Delta &\geqslant & |\widetilde{E}_{\cN}(T_A)|,\nonumber\\
2\,\Delta &\geqslant & |\hat{E}_N(M_A) \cup \widetilde{E}_{\cN}(T_A)|.\label{eq:approx2}
\end{eqnarray} 
As the composite system is structurally controllable and Eqn.~\eqref{eq:approx2} holds,  the output of Algorithm~\ref{alg:twostage} is a $2$-optimal solution to Problem~\ref{prob:int}.
\end{proof}

Now we prove the computational complexity of Algorithm~\ref{alg:twostage}.

\begin{theorem}\label{thm:comp}
Algorithm~\ref{alg:twostage} that takes as input a set of $k$ subsystems and their out-neighbor set and returns a subset of interconnection that guarantee structural controllability of the composite systems has $O(\nT^{2.5})$ complexity.
\end{theorem}
\begin{proof}
Algorithm~\ref{alg:twostage} consists of a minimum weight matching algorithm and a minimum spanning tree algorithm. Construction of the bipartite graph graph has complexity $O(\nT^2)$, where $\nT = \sum_{i=1}^k n_i$ and $n_i$ is the dimension of subsystem $S_i$ for $i=1, \ldots, k$. Minimum weight matching algorithm  has complexity $O(\nT^{2.5})$.

The construction of the directed graph for solving the minimum spanning tree problem has complexity $O(\nT^2)$. A minimum spanning tree algorithm on this graph involves $O(\nT^2)$ computations  \cite{GabGalSpeTar:86}. This proves complexity of Algorithm~\ref{alg:twostage} is $O(\nT^{2.5})$.
\end{proof}
This completes the discussion on approximation algorithm and its complexity.

\begin{rem}
Due to duality between controllability and observability in LTI systems all results of this paper directly follow to the observability problem, where the objective is to find a minimum cardinality set of interconnections among subsystems with pre-specified neighbor set and output matrix that guarantee structural observability of the composite system.
\end{rem}
\subsection{Special Case: Complete Communication Graph}\label{subsec:complete}
A special case of Problem~\ref{prob:int_decision} is when the neighbor set is unconstrained, i.e., a complete communication graph. The communication graph is said to be complete if  any  two distinct subsystems can interconnect each other. In such a case,  the optimal constrained network topology design problem becomes unconstrained. However, hardness result for this case is not known.  We gave polynomial-time algorithms for the unconstrained case  in \cite{MooChaBel:17_homo} and \cite{MooChaBel:17_hetero} for homogeneous (structurally equivalent) and heterogeneous irreducible\footnote{A subsystem $S_i = (\bA_i, \bB_i)$ is said to be irreducible if its digraph $\D(\bA_i)$ is strongly connected.}  subsystems, respectively. While the cases considered in \cite{MooChaBel:17_homo} and \cite{MooChaBel:17_hetero} are polynomial-time solvable, the tractability of the unconstrained case for general  (not irreducible) heterogeneous subsystems is unknown. Algorithm~\ref{alg:twostage} and Theorem~\ref{thm:approx} apply to Problem~\ref{prob:int} when the communication graph is complete (unconstrained) for general subsystem topology and return a $2$-optimal solution.
\subsection{Illustrative Example}\label{subsec:illus} 
In this subsection, we apply Algorithm~\ref{alg:twostage} on a set of subsystems. Consider four subsystems $S_1, S_2, S_3,$ and $S_4$  shown in Figure~\ref{fig:illus_2}. Let the neighbor set of the subsystems are given by $N(S_1) = S_3$, $N(S_2) = S_1,$ $N(S_3) = \{S_2, S_4\}$, and $N(S_4) = \emptyset$.

The bipartite graph $\hat{\B}(\bA_{\T}, \bB_{\T})$ constructed for the subsystems is given in Figure~\ref{fig:bip1}. A minimum weight matching in $\hat{\B}(\bA_{\T}, \bB_{\T})$ is shown in Figure~\ref{fig:bip2}.
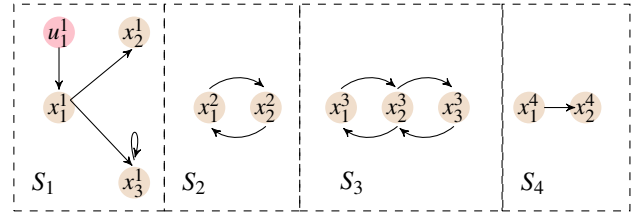
\begin{figure}[h]
\centering
\begin{tikzpicture}[scale=1, ->,>=stealth',shorten >=1pt,auto,node distance=1.85cm, main node/.style={circle,draw,font=\scriptsize\bfseries}]
\definecolor{myblue}{RGB}{80,80,160}
\definecolor{almond}{rgb}{0.94, 0.87, 0.8}
\definecolor{bubblegum}{rgb}{0.99, 0.76, 0.8}
\definecolor{columbiablue}{rgb}{0.61, 0.87, 1.0}

  \fill[bubblegum] (5,1.0) circle (6.0 pt);

  \fill[almond] (5.0,0) circle (6.0 pt);
  \fill[almond] (6,1) circle (6.0 pt);
  \fill[almond] (6, -1) circle (6.0 pt);
  
  \fill[almond] (7.0,0) circle (6.0 pt);
  \fill[almond] (7.75,0) circle (6.0 pt);

  \fill[almond] (8.75,0) circle (6.0 pt);
  \fill[almond] (9.5,0) circle (6.0 pt);
   \fill[almond] (10.25,0) circle (6.0 pt);
   
   \fill[almond] (11.25,0) circle (6.0 pt);
    \fill[almond] (12.0,0) circle (6.0 pt);

   \node at (5,1.0) {\small $u^1_1$};     

  \node at (5,0) {\small $x_1^1$};
  \node at (6,1) {\small $x_2^1$};
  \node at (6,-1) {\small $x_3^1$};

  \node at (7.0,0) {\small $x_1^2$};
  \node at (7.75,0) {\small $x_2^2$};
 
  \node at (8.75,0) {\small $x_1^3$};
  \node at (9.5,0) {\small $x_2^3$};
  \node at (10.25,0) {\small $x_3^3$};
      
    \node at (11.25,0) {\small $x_1^4$};
     \node at (12.0,0) {\small $x_2^4$};

\draw (5.15,0.1)  ->   (6,0.8);  
\draw (5.15,0.1)  ->   (6,-0.8);  
\draw (5.0,0.8)  ->   (5,0.2);  

\draw (11.45,0)  ->   (11.9,0); 

\path[every node/.style={font=\sffamily\small}]
(6,-0.7) edge[loop above] (6,-0.7)
(7.0,0.25) edge[bend left = 40] node [left] {} (7.75,0.25)
(7.75,-0.25) edge[bend left = 40] node [left] {} (7.0,-0.25)

(8.75,0.25) edge[bend left = 40] node [left] {} (9.5,0.25)
(9.5,-0.25) edge[bend left = 40] node [left] {} (8.75,-0.25)
(9.5,0.25) edge[bend left = 40] node [left] {} (10.25,0.25)
(10.25,-0.25) edge[bend left = 40] node [left] {} (9.5,-0.25);

\node (rect) at (5.4,0) [draw,dashed,minimum width=2.0cm,minimum height=2.75cm] {};
\node at (4.8, -1) {$S_1$};
\node (rect) at (7.3,0) [draw,dashed,minimum width=1.8cm,minimum height=2.75cm] {};
\node at (6.8, -1) {$S_2$};
\node (rect) at (9.55,0) [draw,dashed,minimum width=2.7 cm,minimum height=2.75cm] {};
\node at (8.9, -1) {$S_3$};
\node (rect) at (11.74,0) [draw,dashed,minimum width=1.7 cm,minimum height=2.75cm] {};
\node at (11.3, -1) {$S_4$};
\end{tikzpicture}
\caption{\small Structured subsystems $S_1, S_2, S_3,$ and $S_4$ with neighbor sets $N(S_1) = S_3$, $N(S_2) = S_1,$ $N(S_3) = \{S_2, S_4\}$, and $N(S_4) = \emptyset$. We apply Algorithm~\ref{alg:twostage} on this set of subsystems to obtain an optimal topology.}
\label{fig:illus_2}
\end{figure}

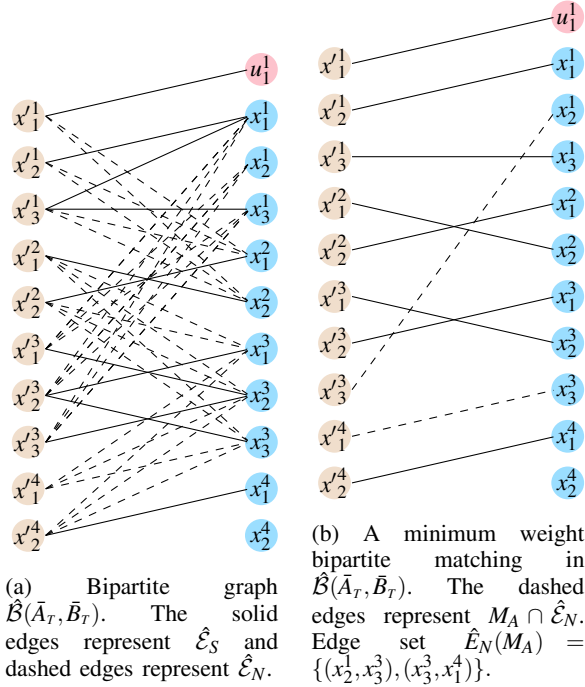
\begin{figure}[h]
\centering
\begin{subfigure}[b]{0.2\textwidth}
\centering
\begin{tikzpicture}[scale=1.55,shorten >=1pt,auto,node distance=1.85cm, main node/.style={circle,draw,font=\scriptsize\bfseries}]
\definecolor{almond}{rgb}{0.94, 0.87, 0.8}
\definecolor{bubblegum}{rgb}{0.99, 0.76, 0.8}
\definecolor{columbiablue}{rgb}{0.61, 0.87, 1.0}
\definecolor{myblue}{RGB}{80,80,160}
\definecolor{mygreen}{RGB}{80,160,80}
\definecolor{myred}{RGB}{144, 12, 63}
          
          \fill[almond] (1,-5) circle (4.0 pt);
          \fill[almond] (1,-5.4) circle (4.0 pt);
          \fill[almond] (1,-5.8) circle (4.0 pt);
          \fill[almond] (1,-6.2) circle (4.0 pt);
          \fill[almond] (1,-6.6) circle (4.0 pt);
          \fill[almond] (1,-7) circle (4.0 pt);
           \fill[almond] (1,-7.8) circle (4.0 pt);
          \fill[almond] (1,-7.4) circle (4.0 pt);
         \fill[almond] (1,-8.2) circle (4.0 pt);
          \fill[almond] (1,-8.6) circle (4.0 pt);
          
          \node at (1,-5) {\small ${x'}_1^1$};
          \node at (1,-5.4) {\small ${x'}_2^1$};
          \node at (1,-5.8) {\small ${x'}_3^1$};
          \node at (1,-6.2) {\small ${x'}_1^2$};
          \node at (1,-6.6) {\small ${x'}_2^2$};
          \node at (1,-7) {\small ${x'}_1^3$};
          \node at (1,-7.4) {\small ${x'}_2^3$};
          \node at (1,-7.8) {\small ${x'}_3^3$};
         \node at (1,-8.2) {\small ${x'}_1^4$};
          \node at (1,-8.6) {\small ${x'}_2^4$};
          
          \fill[bubblegum] (3,-4.6) circle (4.0 pt);   

          \fill[columbiablue] (3,-5) circle (4.0 pt);
          \fill[columbiablue] (3,-5.4) circle (4.0 pt);
          \fill[columbiablue] (3,-5.8) circle (4.0 pt);
          \fill[columbiablue] (3,-6.2) circle (4.0 pt);
          \fill[columbiablue] (3,-6.6) circle (4.0 pt);
          \fill[columbiablue] (3,-7) circle (4.0 pt);
          \fill[columbiablue] (3,-7.4) circle (4.0 pt);
          \fill[columbiablue] (3,-7.8) circle (4.0 pt);
          \fill[columbiablue] (3,-8.2) circle (4.0 pt);
          \fill[columbiablue] (3,-8.6) circle (4.0 pt);
     
          \node at (3,-4.6) {\small $u^1_1$};    
          \node at (3,-5) {\small $x_1^1$};
          \node at (3,-5.4) {\small $x_2^1$};
          \node at (3,-5.8) {\small $x_3^1$};
          \node at (3,-6.2) {\small $x_1^2$};
          \node at (3,-6.6) {\small $x_2^2$};
          \node at (3,-7) {\small $x_1^3$};
          \node at (3,-7.4) {\small $x_2^3$};
          \node at (3,-7.8) {\small $x_3^3$};
          
           \node at (3,-8.2) {\small $x_1^4$};
          \node at (3,-8.6) {\small $x_2^4$};
          
          \draw[] (1.15,-5) -- (2.9,-4.6);
          \draw [](1.15,-5.4) -- (2.9,-5.0);
          \draw [](1.15,-5.8) -- (2.9,-5.0);
          \draw [](1.15,-5.8) -- (2.9,-5.8);
          
           \draw[] (1.15,-6.2) -- (2.9,-6.6);
          \draw[] (1.15,-6.6) -- (2.9,-6.2);
          
           \draw[] (1.15,-7) -- (2.9,-7.4);
          \draw[] (1.15,-7.4) -- (2.9,-7);
         \draw[] (1.15,-7.4) -- (2.9,-7.8);
          \draw[] (1.15,-7.8) -- (2.9,-7.4);
          
          \draw[] (1.15,-8.6) -- (2.9,-8.2);

        \draw[dashed] (1.15,-7) -- (2.9,-5);
          \draw [dashed](1.15,-7.4) -- (2.9,-5.4);
          \draw [dashed](1.15,-7.8) -- (2.9,-5.8);
          
                  \draw[dashed] (1.15,-7) -- (2.9,-5);
                   \draw[dashed] (1.15,-7.4) -- (2.9,-5);
                    \draw[dashed] (1.15,-7.8) -- (2.9,-5);
                    \draw [dashed](1.15,-7.0) -- (2.9,-5.4);
          \draw [dashed](1.15,-7.4) -- (2.9,-5.4);
          \draw [dashed](1.15,-7.8) -- (2.9,-5.4);
           \draw [dashed](1.15,-7.0) -- (2.9,-5.8);
          \draw [dashed](1.15,-7.4) -- (2.9,-5.8);
          \draw [dashed](1.15,-7.8) -- (2.9,-5.8);
          
              \draw[dashed] (1.15,-8.2) -- (2.9,-7);
          \draw[dashed] (1.15,-8.2) -- (2.9,-7.4);
         \draw[dashed] (1.15,-8.2) -- (2.9,-7.8);
          \draw[dashed] (1.15,-8.6) -- (2.9,-7);
          \draw[dashed] (1.15,-8.6) -- (2.9,-7.4);
         \draw[dashed] (1.15,-8.6) -- (2.9,-7.8);
         
               \draw[dashed] (1.15,-5) -- (2.9,-6.2);
                \draw[dashed] (1.15,-5.4) -- (2.9,-6.2);
                 \draw[dashed] (1.15,-5.8) -- (2.9,-6.2);
                 \draw[dashed] (1.15,-5) -- (2.9,-6.6);
                \draw[dashed] (1.15,-5.4) -- (2.9,-6.6);
                 \draw[dashed] (1.15,-5.8) -- (2.9,-6.6);
                 
           \draw[dashed] (1.15,-6.2) -- (2.9,-7);
          \draw[dashed] (1.15,-6.2) -- (2.9,-7.4);
         \draw[dashed] (1.15,-6.2) -- (2.9,-7.8);
          \draw[dashed] (1.15,-6.6) -- (2.9,-7);
          \draw[dashed] (1.15,-6.6) -- (2.9,-7.4);
         \draw[dashed] (1.15,-6.6) -- (2.9,-7.8);
\end{tikzpicture}
\caption{\small Bipartite graph $\hat{\B}(\bA_{\T}, \bB_{\T})$. The solid edges represent $\hat{\E}_{S}$ and dashed edges represent $\hat{\E}_{N}$.}\label{fig:bip1}
\end{subfigure}~\hspace{2 mm}
\begin{subfigure}[b]{0.2\textwidth}
\centering
\begin{tikzpicture}[scale=1.55,shorten >=1pt,auto,node distance=1.85cm, main node/.style={circle,draw,font=\scriptsize\bfseries}]
\definecolor{almond}{rgb}{0.94, 0.87, 0.8}
\definecolor{bubblegum}{rgb}{0.99, 0.76, 0.8}
\definecolor{columbiablue}{rgb}{0.61, 0.87, 1.0}
\definecolor{myblue}{RGB}{80,80,160}
\definecolor{mygreen}{RGB}{80,160,80}
\definecolor{myred}{RGB}{144, 12, 63}
          
          \fill[almond] (1,-5) circle (4.0 pt);
          \fill[almond] (1,-5.4) circle (4.0 pt);
          \fill[almond] (1,-5.8) circle (4.0 pt);
          \fill[almond] (1,-6.2) circle (4.0 pt);
          \fill[almond] (1,-6.6) circle (4.0 pt);
          \fill[almond] (1,-7) circle (4.0 pt);
           \fill[almond] (1,-7.8) circle (4.0 pt);
          \fill[almond] (1,-7.4) circle (4.0 pt);
         \fill[almond] (1,-8.2) circle (4.0 pt);
          \fill[almond] (1,-8.6) circle (4.0 pt);
          
          \node at (1,-5) {\small ${x'}_1^1$};
          \node at (1,-5.4) {\small ${x'}_2^1$};
          \node at (1,-5.8) {\small ${x'}_3^1$};
          \node at (1,-6.2) {\small ${x'}_1^2$};
          \node at (1,-6.6) {\small ${x'}_2^2$};
          \node at (1,-7) {\small ${x'}_1^3$};
          \node at (1,-7.4) {\small ${x'}_2^3$};
          \node at (1,-7.8) {\small ${x'}_3^3$};
         \node at (1,-8.2) {\small ${x'}_1^4$};
          \node at (1,-8.6) {\small ${x'}_2^4$};
          
          \fill[bubblegum] (3,-4.6) circle (4.0 pt);   

          \fill[columbiablue] (3,-5) circle (4.0 pt);
          \fill[columbiablue] (3,-5.4) circle (4.0 pt);
          \fill[columbiablue] (3,-5.8) circle (4.0 pt);
          \fill[columbiablue] (3,-6.2) circle (4.0 pt);
          \fill[columbiablue] (3,-6.6) circle (4.0 pt);
          \fill[columbiablue] (3,-7) circle (4.0 pt);
          \fill[columbiablue] (3,-7.4) circle (4.0 pt);
          \fill[columbiablue] (3,-7.8) circle (4.0 pt);
          \fill[columbiablue] (3,-8.2) circle (4.0 pt);
          \fill[columbiablue] (3,-8.6) circle (4.0 pt);
     
          \node at (3,-4.6) {\small $u^1_1$};    
          \node at (3,-5) {\small $x_1^1$};
          \node at (3,-5.4) {\small $x_2^1$};
          \node at (3,-5.8) {\small $x_3^1$};
          \node at (3,-6.2) {\small $x_1^2$};
          \node at (3,-6.6) {\small $x_2^2$};
          \node at (3,-7) {\small $x_1^3$};
          \node at (3,-7.4) {\small $x_2^3$};
          \node at (3,-7.8) {\small $x_3^3$};
          
           \node at (3,-8.2) {\small $x_1^4$};
          \node at (3,-8.6) {\small $x_2^4$};
          
          \draw[] (1.15,-5) -- (2.9,-4.6);       
          \draw [](1.15,-5.4) -- (2.9,-5.0);
           \draw [](1.15,-5.8) -- (2.9,-5.8);
           
           \draw[] (1.15,-6.2) -- (2.9,-6.6);
          \draw[] (1.15,-6.6) -- (2.9,-6.2);
          
           \draw[] (1.15,-7) -- (2.9,-7.4);
          \draw[] (1.15,-7.4) -- (2.9,-7);
          
          \draw[] (1.15,-8.6) -- (2.9,-8.2);

          \draw [dashed](1.15,-7.8) -- (2.9,-5.4);
          
         \draw[dashed] (1.15,-8.2) -- (2.9,-7.8);

\end{tikzpicture}
\caption{\small A minimum weight bipartite matching in $\hat{\B}(\bA_{\T}, \bB_{\T})$. The dashed edges represent $M_A \cap \hat{\E}_{N}$. Edge set $\hat{E}_N(M_A) = \{ (x_2^1, x_3^3), (x_3^3, x_1^4) \}.$}\label{fig:bip2}
\end{subfigure}
\caption{\small Implementation of minimum weight matching algorithm (Stage~1 of Algorithm~\ref{alg:twostage}) for the subsystems given in Figure~\ref{fig:illus_2}.}
\end{figure}

Now we construct the directed graph $\pT_{\cN}$ shown in Figure~\ref{fig:illus_3}. For the given subsystems the vertex set in $\pT_{\cN}$ is $\{\U, \cN_1, \ldots, \cN_7 \}$, where $\cN_1 =x_1^1$, $\cN_2 =x_2^1$, $\cN_3 =x_3^1$, $\cN_4 = S_2$, $\cN_5 =S_3$, $\cN_6 =x_1^4$, and $\cN_7 =x_2^4$. A minimum spanning tree of $\pT_{\cN}$ rooted at $\U$ is shown in Figure~\ref{fig:illus_4}.
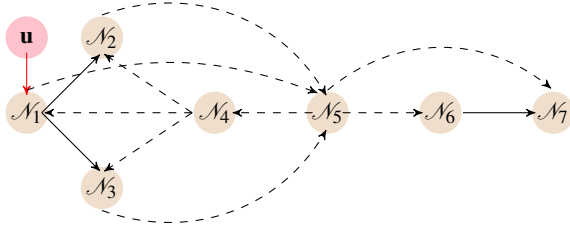
\begin{figure}[h]
\centering
\begin{tikzpicture}[->,>=stealth',shorten >=1pt,auto,node distance=1.85cm, main node/.style={circle,draw,font=\scriptsize\bfseries}]
\definecolor{myblue}{RGB}{80,80,160}
\definecolor{almond}{rgb}{0.94, 0.87, 0.8}
\definecolor{bubblegum}{rgb}{0.99, 0.76, 0.8}
\definecolor{columbiablue}{rgb}{0.61, 0.87, 1.0}

  \fill[bubblegum] (5,1.0) circle (8.0 pt);

  \fill[almond] (5.0,0) circle (8.0 pt);
  \fill[almond] (6,1) circle (8.0 pt);
  \fill[almond] (6, -1) circle (8.0 pt);
  
  \fill[almond] (7.5,0) circle (8.0 pt);

  \fill[almond] (9.0,0) circle (8.0 pt);

   \fill[almond] (10.5,0) circle (8.0 pt);
    \fill[almond] (12.0,0) circle (8.0 pt);

   \node at (5,1.0) {\small $\U$};     

  \node at (5,0) {\small $\cN_1$};
  \node at (6,1) {\small $\cN_2$};
  \node at (6,-1) {\small $\cN_3$};

  \node at (7.5,0) {\small $\cN_4$};

  \node at (9,0) {\small $\cN_5$};
      
    \node at (10.5,0) {\small $\cN_6$};
     \node at (12.0,0) {\small $\cN_7$};

\draw (5.2,0)  ->   (6,0.8);  
\draw (5.2,0)  ->   (6,-0.8);  

 \draw[dashed] (7.2,0) -> (5.2,0); 
  \draw[dashed] (7.2,0) -> (6,0.8); 
   \draw[dashed] (7.2,0) -> (6,-0.8); 
   \draw[dashed] (9.2,0) -> (10.3,-0.0); 
   
   \draw[dashed] (8.8,0) -> (7.7,-0.0); 
     
   \draw[red] (5,0.8) -> (5,0.2); 
   
    \draw[] (10.8,0) -> (11.8,0); 

\path[every node/.style={font=\sffamily\small}]
(9,0.3) edge[dashed, bend left = 40] node [left] {} (12,0.3)
(6, 1.3) edge[dashed,bend left = 40] node [left] {} (9,0.2)
(6, -1.3) edge[dashed,bend right = 40] node [right] {} (9,-0.2)

(5,0.3) edge[dashed,bend left = 20] node [left] {} (8.9,0.25);
\end{tikzpicture}
\caption{\small Directed graph $\pT_{\cN}$ with vertex set SCC's of subsystems and $\U$. Solid edges represent $E^1_{\cN}$, dashed edges represent $E^2_{\cN}$, and the red edge represent $E^3_{\cN}$.}
\label{fig:illus_3}
\end{figure}

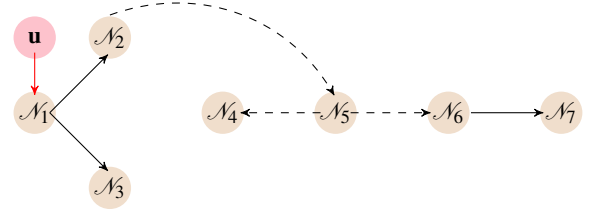
\begin{figure}[h]
\centering
\begin{tikzpicture}[->,>=stealth',shorten >=1pt,auto,node distance=1.85cm, main node/.style={circle,draw,font=\scriptsize\bfseries}]
\definecolor{myblue}{RGB}{80,80,160}
\definecolor{almond}{rgb}{0.94, 0.87, 0.8}
\definecolor{bubblegum}{rgb}{0.99, 0.76, 0.8}
\definecolor{columbiablue}{rgb}{0.61, 0.87, 1.0}

  \fill[bubblegum] (5,1.0) circle (8.0 pt);

  \fill[almond] (5.0,0) circle (8.0 pt);
  \fill[almond] (6,1) circle (8.0 pt);
  \fill[almond] (6, -1) circle (8.0 pt);
  
  \fill[almond] (7.5,0) circle (8.0 pt);

  \fill[almond] (9.0,0) circle (8.0 pt);

   \fill[almond] (10.5,0) circle (8.0 pt);
    \fill[almond] (12.0,0) circle (8.0 pt);

   \node at (5,1.0) {\small $\U$};     

  \node at (5,0) {\small $\cN_1$};
  \node at (6,1) {\small $\cN_2$};
  \node at (6,-1) {\small $\cN_3$};

  \node at (7.5,0) {\small $\cN_4$};

  \node at (9,0) {\small $\cN_5$};
      
    \node at (10.5,0) {\small $\cN_6$};
     \node at (12.0,0) {\small $\cN_7$};

\draw (5.2,0)  ->   (6,0.8);  
\draw (5.2,0)  ->   (6,-0.8);

   \draw[dashed] (9.2,0) -> (10.3,-0.0); 
   
   \draw[dashed] (8.8,0) -> (7.7,-0.0); 
     
   \draw[red] (5,0.8) -> (5,0.2); 
   
    \draw[] (10.8,0) -> (11.8,0); 

\path[every node/.style={font=\sffamily\small}]
(6, 1.3) edge[dashed,bend left = 40] node [left] {} (9,0.2);
\end{tikzpicture}
\caption{\small Minimum spanning tree $T_A$ of $\pT_{\cN}$. Dashed edges represent $T_A \cap E^2_{\cN}$. Edge set $\widetilde{E}_{\cN}(T_A) = \{ (x_2^1, x_1^3), (x_1^3, x_1^2), (x_1^3, x_1^4)  \}$.}
\label{fig:illus_4}
\end{figure}
The interconnection edge set obtained as output of Algorithm~\ref{alg:twostage} is  $\{ (x_2^1, x_3^3), (x_3^3, x_1^4), (x_2^1, x_1^3), (x_1^3, x_1^2), (x_1^3, x_1^4)  \}$.
\section{Extensions}\label{sec:ext}
In this section, we discuss possible extensions of the proposed algorithm and results. The cases considered are (a)~weighted optimal constrained network topology design problem and (b)~switched linear systems. Weights of interconnection links reflect the   installation and monitoring cost of interaction links between different subsystems considering system specific constraints like delay and spatial locations. 

\subsection{Weighted Optimal Constrained Interconnection Problem}\label{subsec:weight}
The weighted constrained optimal  network topology design problem is as follows: given a set of subsystems $S_1, \ldots, S_k$, neighbor set of each subsystem $N(S_i)$, for $i=1, \ldots, k$, and a weight function that associate a weight with each interconnection link, say interconnection link $(x^i_p, x^j_q)$ is associated with weight $c_{I}(x^i_p, x^j_q)$ for all $p \in \{1, \ldots, n_i\}$, $q\in \{1, \ldots, n_j\}$ and $i \neq j$ satisfying $S_j \in N(S_i)$, find a minimum weight optimal network topology design of the composite system that is structurally controllable.
We now give the following result.

\begin{cor}\label{cor:approx_ext}
Consider $k$ structured subsystems $(\bA_i, \bB_i)$, where $\bA_i  \in \{0,\*\}^{n_i \times n_i}$ and $\bB_i \in \{0, \*\}^{n_i \times m_i}$, the neighbor set $N(S_i)$, for $i  \in \{1,\ldots, k\}$, and the weight associated with each possible interconnection link. Then,\\ (i)~constrained optimal network topology design problem with weights for interconnection links is NP-hard, and \\
(ii)~Algorithm~\ref{alg:twostage} with modified $w_{\sB}$ and $w_{\sT}$ returns a set of interconnections such that the composite system obtained using the interconnections is a $2$-optimal solution to the weighted version of Problem~\ref{prob:int}.
\end{cor}
The proof of NP-hardness follows as the special case with all weights  uniform is NP-hard.  Algorithm~\ref{alg:twostage} and Theorem~\ref{thm:approx} when extended to the weighted case by modifying the weight functions $w_{\sB}$ and $w_{\sT}$, so as to incorporate the weight associated with the interconnections, returns a $2$-optimal solution to the weighted version of Problem~\ref{prob:int}.

\subsection{Switched Linear Systems}\label{subsec:switch}
Switched linear system is a special  class of hybrid control systems consisting of several linear systems and a rule that determines the switching between them (see \cite{SunGeLee-02} and \cite{LiuLinChe-13} for more details).  We consider structured subsystems $S_1, \ldots, S_k$ such that the neighbor set of the subsystems at different instants of time differs on account of communication limitations. This occurs when simultaneous interactions between subsystems are difficult due to interference constraint, however, subsystems can interact at distinct instants of time.  In other words, we consider a  switched LTI system where each switching mode corresponds to a specified neighbor set for each subsystem. Let 
\begin{equation}
\dot{x}(t) = \bA_{\T}^{\sigma(t)}x(t) + \bB_{\T}^{\sigma(t)}u(t),
\end{equation}
where $\sigma: \R_{+} \rightarrow \mathbb{M} \equiv \{1, \ldots, z\}$ is a switching signal such that $( \bA_{\T}^{\sigma(t)}, \bB_{\T}^{\sigma(t)})$ consists of $z$ modes of the composite system  and $\sigma(t) = j$ means that the pair $(\bA_{\T}^j, \bB_{\T}^j)$ is active at instant $t$. It is assumed that when all possible interconnections are established at each mode, then there exists a switching (scheduling) sequence that guarantee controllability of the switched system.  Our objective is to find a minimum cardinality subset of interconnection links to establish among the subsystems in the different modes such that the composite switched system is structurally controllable \cite{PeqPap-17}. 

\begin{cor}\label{cor:approx_switch}
Consider $k$ structured subsystems $(\bA_i, \bB_i)$, where $\bA_i  \in \{0,\*\}^{n_i \times n_i}$ and $\bB_i \in \{0, \*\}^{n_i \times m_i}$, the neighbor set $N_{\mathbb{M}}(S_i)$, for $i  \in \{1,\ldots, k\}$ and $\mathbb{M}=1, \ldots, z$. Then,\\ (i)~constrained optimal network topology design problem is NP-hard for switched linear systems, and \\
(ii)~Algorithm~\ref{alg:twostage} on the union graph of the different modes of the structured system returns a set of interconnections such that the composite system obtained using the interconnections is a $2$-optimal solution to the constrained optimal network topology design problem in switched linear systems.
\end{cor}
The proof of NP-hardness follows as the special case where there is only one mode, i.e., $ \mathbb{M}=1$, is NP-hard. The results and Algorithm~\ref{alg:twostage} can be extended to the switched linear systems case after taking the union of the digraphs of the systems corresponding to different modes using Theorem~4 in \cite{LiuLinChe-13}.
\section{Conclusion}\label{sec:conclu}
This paper dealt with controllability of complex systems referred as composite systems consisting of many subsystems or agents interconnected to perform some desired task. The analysis is done in a structured framework by using the sparsity pattern of the system matrices. In this paper, we addressed structural controllability of an LTI composite system consisting of several subsystems. Given a set of subsystems and a pre-specified set of out-neighbors of each subsystem, where each subsystem is not necessarily structurally controllable, the objective is to find a minimum cardinality set of interconnections among these subsystems such that the composite system is structurally controllable using the specified input matrix. This problem is referred as the optimal constrained network topology design problem. We first proved that  optimal constrained network topology design problem is NP-hard, using polynomial-time reduction from degree constrained spanning tree problem  (Theorem~\ref{thm:NP}). Then we gave a polynomial-time approximation algorithm to solve the optimal constrained network topology design problem (Algorithm~\ref{alg:twostage}). This algorithm consists of a minimum weight matching algorithm, that guarantee the no-dilation condition for structural controllability, and a minimum spanning tree problem, that guarantee accessibility condition. We proved that the proposed algorithm obtaines a $2$-optimal solution to the optimal constrained network topology design problem (Theorem~\ref{thm:approx}). We extended the approach given in this paper to the case of weighted optimal constrained network topology design (Corollary~\ref{cor:approx_ext}) and constrained optimal network topology design in switched linear systems (Corollary~\ref{cor:approx_switch}).  Needless to elaborate, due to duality between controllability and observability in LTI systems all results of this paper directly follow to the observability problem. 
\bibliographystyle{myIEEEtran}  
\bibliography{myreferences}
\end{document}